\documentclass[12pt]{amsart}
\usepackage{amssymb,latexsym}
\usepackage{amsthm}
\usepackage{amsfonts}
\usepackage{amssymb}
\usepackage{fullpage}
\usepackage{enumerate}
\newtheorem{theorem}{Theorem}
\newtheorem{lemma}[theorem]{Lemma}
\theoremstyle{definition}
\newtheorem{definition}[theorem]{Definition}

\newtheorem{proposition}[theorem]{Proposition}

\newtheoremstyle{remark}{9pt}{9pt}{}{0pt}{\bf}{.}{0.5em}{}
\theoremstyle{remark} \newtheorem{remark}[theorem]{Remark}

\newcommand{\MM}{{\mathcal M}}
\newcommand{\OO}{{\mathcal O}}
\newcommand{\TT}{{\mathcal T}}
\newcommand{\CC}{{\mathbb C}}
\newcommand{\DD}{{\mathbb D}}
\newcommand{\ZZ}{{\mathbb Z}}
\newcommand{\GG}{{\mathbb G}}

\def\tr{\operatorname{tr}}
\def\aut{\operatorname{Aut}}
\def\tr{\operatorname{tr}}

\def\div{\operatorname{div}}
\def\Ker{\operatorname{Ker}}
\title{A local form of the automorphisms of the spectral ball}
\author{\L ukasz Kosi\'nski}
\address{ Wydzia\l Matematyki i informatyki, Uniwersytet Jagiello\'nski\\ ul. prof. St. \L ojasiewicza~6\\ 30-348 Krak\'ow,}
\email{lukasz.kosinski@gazeta.pl}
\thanks{}
\subjclass[2010]{}
\begin{document}
\maketitle

\begin{abstract} We show that the group generated by by triangular and diagonal conjugations is dense in $\aut(\Omega_2)$ (in the compact-uniform topology). Moreover, it is shown that any automorphism of $\Omega_2$ is a local holomorphic conjugation.
\end{abstract}

\section{Introduction and statement of the result}

Let $\Omega_n$ denote the spectral ball in $\CC^{n^2}$, that is a domain composed of $n\times n$ complex matrices whose spectral radius is less then $1.$

The natural question that arises is to classify its group of automorphisms. It may be easily checked that among them there are the following three forms:
\begin{enumerate}[(i)]
\item {\it Transposition}: $\tau: x\mapsto x^t,$
\item {\it M\"obius maps}: $m_{\alpha,\gamma}: x\mapsto \gamma(x-\alpha)(1-\overline{\alpha}x)^{-1},$ where $\alpha$ lies in the unit disc and $|\gamma|=1,$
\item {\it Conjugations:} $\mathcal J_u:x\mapsto u(x)^{-1}xu(x),$ where $u:\Omega\to \MM_n^{-1}$ is a conjugate invariant holomorphic map, i.e. $u(q^{-1}xq)=u(x)$ for each $x\in \Omega$ and $q\in \mathcal M_n^{-1}$
\end{enumerate} (throughout the paper $\MM^{-1}_n$ denotes the group of invertible $n\times n$ complex matrices).

Ransford and White have asked in \cite{Ran-Whi} whether the compositions of the three above forms generate the whole $\aut (\Omega)$ - the group of automorphisms of the spectral ball. In \cite{Kos} we have shown that the answer to this question is negative. Nevertheless the question about the description of $\aut(\Omega)$ remains open. In this note we deal with this problem.

Let us introduce some notation. $\mathcal M_n$ denotes the algebra of $n\times n$ matrices with complex coefficients. Moreover, $\mathcal T_n$ is the set of non-cyclic matrices lying in $\Omega_n$. Recall that a matrix $M$ is cyclic if there exists a cyclic vector $v\in \mathbb C^n$ for $M$, i.e. $v$ such that $(v,Mv,\ldots , M^{n-1} v)$ spans $\mathbb C^n.$ It is well known that $M$ is cyclic if and only if for any $\lambda\in \mathbb C$ $$\dim \Ker (M-\lambda)\leq 1.$$ In particular, any matrix having $n$ distinc eigenvalues is cyclic.

In the case $n=2$ we shall simply write $\Omega=\Omega_2$, $\MM=\mathcal M_2$ and $\mathcal T=\mathcal T_2.$

\begin{definition}\label{def} Let $U$ be an open subset of the spectral ball $\Omega.$ We shall say that a mapping $\varphi:U\to \mathcal M_n$ is a \emph{holomorphic conjugation} if there is a holomorphic mapping $p:U\mapsto \mathcal M_n^{-1}$ such that $\varphi(x)=p(x)xp(x)^{-1},$ $x\in U.$

Moreover, we shall say that $\varphi$ is a \emph{local holomorphic conjugation} if any $x\in U$ has a neighborhood $V$ such that $\varphi$ restricted to it is a holomorphic conjugation.
\end{definition}

The paper is organized as follows. Recall that we presented in \cite{Kos} two counterexamples to the question on the description of the group of automorphisms of the spectral ball. We focus on them in the sense that we classify all automorphisms of the spectral ball of the form $$\Omega\ni x\mapsto u(x)xu(x)^{-1}\in \Omega,$$ where $u\in \OO(\Omega,\MM^{-1})$ is such that $u(x)$ is either diagonal or triangular, $x\in \Omega$ (we shall call the \emph{diagonal} and \emph{triangular} conjugations respectively).

It is known that any automorphism of the spectral ball fixing the origin preserves the spectrum. Therefore, trying to describe the group $\aut(\Omega)$ it is natural to investigate the behavior of automorphisms restricted to the fibers of $\Omega,$ i.e. sets of the form $\mathcal F_{(\lambda_1,\lambda_2)}:=\{x\in \Omega:\ \sigma(x)=\{\lambda_1,\lambda_2\}\},$ $\lambda_1,\lambda_2\in \DD,$ where $\sigma(x)$ denotes the spectrum of $x\in \MM.$  If $\lambda_1\neq \lambda_2$, the fiber $\mathcal F_{(\lambda_1,\lambda_2)}$ forms a submanifold known as the \emph{Danielewski surfaces}. Recall that the Danielewski surface associated with a polynomial $p\in \CC[z]$ is given by $$D_p:=\{(x,y,z)\in \CC^3:\ xy=p(x)\}$$ and its complex structure is naturally induced from $\CC^3.$ Algebraic properties of Danielewski surfaces have been intensively studied in the literature. Anyway a little is known about their holomorphic automorphisms. It was lastly shown (see \cite{KK} and \cite{Lin}) that the group generated by shears and overshears (for the definition see e.g. \cite{Linphd}) is dense in the group of holomorphic automorphisms. It was a little surprise to us that shear are just the restriction to the fiber of diagonal and triangular conjugations.

Following the idea idea from \cite{Linphd} (we recall all details for the convenience of the reader) we shall show that the spectral ball satisfies the property obtained by Anders\'en and Lempert for holomorphic automorphisms of $\CC^n$ (see \cite{And}, \cite{And-Lem} and \cite{For-Ros}). To be more precise we shall show that the group generated by triangular and diagonal conjugations is dense in $\aut(\Omega)$ (in the local-uniform topology).

Finally, we shall show that the uniform limit of conjugations is a local conjugation in a neighborhood of $\TT$. This, together with the density of the the group generated by triangular and diagonal conjugations and the results of P.J.~Thomas and J.~Rostand (see \cite{Pas} and \cite{Ros}) imply that any automorphism of the spectral ball is a local conjugation.

\section{Diagonal and triangular conjugations}
Let us focus on conjugations of the following form:
\begin{equation}\label{diag}
 \tilde D_a x\mapsto \left(\begin{array}{cc} a(x) & 0 \\ 0 & 1/a(x) \\\end{array} \right) x \left(\begin{array}{cc} a(x) & 0 \\ 0 & 1/a(x) \\\end{array} \right)^{-1},
\end{equation}
 and

\begin{equation}\label{diag1}
T_b:x\mapsto \left(\begin{array}{cc} 1 & 0 \\ b(x) & 1 \\\end{array} \right) x \left(\begin{array}{cc} 1 & 0 \\ b(x) & 1 \\\end{array} \right)^{-1}.
\end{equation} Out aim is to describe $a$ and $b$ such that the above mappings are automorphisms of the spectral unit ball.

The case (1) is easy. First note that the simply-connectedness of $\Omega$ imply that there is $\tilde a\in \mathcal O(\Omega)$ such that $a=\exp(\tilde a/2).$ Moreover, for fixed $x_{11}$ and $x_{22}$ the mapping $(x_{12},x_{21})\mapsto (\exp(\tilde a (x)) x_{12},\exp(-\tilde a(x))x_{21})$ is injective on its domain. In particular, the mapping $$z\mapsto \exp(\tilde a(x_{11},z,t/z,x_{22}))z$$ is an automorphism of $\mathbb C_*$ for $t$ sufficiently small. Therefore $z\mapsto \tilde a(x_{11},z,t/z,x_{22})$ is constant, whence $\tilde a$ depends only on $x_{11},$ $x_{22}$ and $x_{12}x_{21}.$

\medskip

Now we focus our attention on \eqref{diag1}. We want to find $b$ such that $$x\mapsto \left(\begin{array}{cc} x_{11}-b(x)x_{12} & x_{12} \\ b(x)x_{11}+x_{21}-b^2(x)x_{12}+b(x)x_{22} & b(x)x_{12}+x_{22} \\\end{array} \right)$$ is an automorphism of $\Omega.$ Putting $(s,p):=(\tr x,\det x)\in \GG_2$ (throughout the paper $\GG_2$ denotes the symmetrized bidisc - see e.g. \cite{Edi-Zwo}) and looking at the automorphism restricted to the fibers (it is obvious that conjugations preserve fibers) we get that the mapping $$(x_{11},x_{12})\mapsto (x_{11}- x_{12} b(\left(\begin{array}{cc} x_{11} & x_{12} \\ (x_{11}(s-x_{11})-p)x_{12}^{-1}  & s-x_{11} \\\end{array} \right)),x_{12})$$ is an automorphism of $\mathbb C\times \mathbb C_*$. It is quite easy to observe that if $(x,y)\mapsto (x-f(x,y),y)$ is an automorphism of $\mathbb C\times \mathbb C_*$, iff $f(x,y)=x(1-c(y))-\gamma(y)$, $x\in \CC,$ $y\in \CC_*$, where $c\in \mathcal O(\mathbb C_*,\mathbb C_*)$ and $\gamma\in \mathcal O(\mathbb C_*,\mathbb C)$.

Applying this reasoning to $b$ we simple find that there are $c\in \OO(\CC_*\times\GG_2,\CC_*)$ and $\gamma\in \OO(\CC_*\times\GG_2)$ such that $$b(x)=x_{11}\frac{1-c(x_{12},\tr x,\det x)}{x_{12}}+\frac{\gamma(x_{12},\tr x,\det x)}{x_{12}}.$$ Putting $x_{11}=0$ we see that $\gamma(x_{12},s,p)=x_{12}\beta(x_{12},s,p),$ where $\beta\in \mathcal O(\mathbb C\times \mathbb G_2).$ Using this we see that $c$ may be extend holomorphically through $x_{12}=0$. Moreover, one may easily check that $c(x_{12},s,p)=e^{x_{12}\alpha(x_{12},\tr x,\det x)},$ where $\alpha\in \mathcal O(\mathbb C\times \mathbb G_2)$. Thus \begin{equation}\label{b} b(x)=x_{11}\frac{1-\exp(x_{12}\alpha(x_{12},\tr x,\det x))}{x_{12}}+\beta(x_{12},\tr x,\det x).\end{equation}

\begin{remark}
Note that $T_b$ is generated by $T_{\beta}$ and $T_{\gamma},$ $\gamma(x)=x_{11}\frac{1-\exp(x_{12}\alpha(x_{12},\tr x,\det x))}{x_{12}},$ where $\alpha$ and $\beta$ satisfy \eqref{b}.
\end{remark}

\section{Vector fields generated by triangular and diagonal conjugations and relations between them}\label{3}

To describe the vector fields generated by \eqref{diag} and \eqref{diag1} it is convenient to introduce the following notation:

$D_a:=\tilde D_{\exp(a/2)}.$

$T_{\alpha}':=T_b,$ where $b(x)=x_{11}\frac{1-\exp(x_{12}\alpha(x_{12},\tr x,\det x))}{x_{12}}.$

Note that the mappings $D_a,$ $T_{\beta}$ and $T'_{\alpha}$ generate the following vector fields $\frac{d}{dt}\Phi_t(\Phi_{t_0}^{-1}(x))|_{t=t_0},$ where $\{\Phi_t\}$ is one of one-parameter groups $\{ D_{ta}\},$ $\{T_{t\beta}\}$, $\{T'_{t\alpha}\}$ (such vector fiels are called sometimes infinitesimal generators):

\begin{align*}
HD_a &=ax_{12}\frac{\partial }{\partial x_{12}}-ax_{21} \frac{\partial }{\partial x_{21}},\\
HT_{\beta}&=-\beta x_{12}\frac{\partial }{\partial x_{11}}+(\beta x_{11}-\beta^2x_{12}+\beta x_{22})\frac{\partial }{\partial x_{21}}+\beta x_{12} \frac{\partial }{\partial x_{22}},\\
HT'_{\alpha}&=x_{11}x_{12}\alpha\frac{\partial }{\partial x_{11}}+(x_{22}-x_{11})x_{11}\alpha\frac{\partial }{\partial x_{21}}-x_{11}x_{22} \alpha\frac{\partial }{\partial x_{22}}.
\end{align*}

Additionally $\tau \circ T_{\beta}\circ \tau $ and $\tau\circ T'_{\alpha}\circ \tau$ (where $\tau(x)=x^t$ is a transposition) generate

\begin{align*} \tilde HT_{\beta}&=-\beta x_{21}\frac{\partial }{\partial x_{11}}+(bx_{11}-b^2x_{21}+bx_{22})\frac{\partial }{\partial x_{12}}+bx_{21} \frac{\partial }{\partial x_{22}},\\
\tilde HT'_{\alpha}&=x_{11}x_{21}\alpha\frac{\partial }{\partial x_{11}}+(x_{22}-x_{11})x_{11}\alpha\frac{\partial }{\partial x_{12}}-x_{11}x_{22} \alpha\frac{\partial }{\partial x_{22}}.
\end{align*}

Let us write the above vector fields in ''spectral'' coordinates $$(x,y,s,p)=(x_{11},x_{12},\tr x,\det x).$$ Observe that for any vector field $V$ on $\Omega$ orthogonal to $\det x$ and $\tr x$ its divergence $\div V$ (in Euclidean coordinates) is equal to $\frac{\partial v_1}{\partial x}+\frac{\partial v_2}{\partial y}-\frac{v_2}{y},$ where $V=v_1\frac{\partial}{\partial x}+v_2\frac{\partial}{\partial y}$. We get:

\begin{equation}HD_a=ay\frac{\partial}{\partial y},\ HT_{\beta}=-\beta y\frac{\partial}{\partial x}\quad \text{and}\quad HT'_{\beta}=xy\beta \frac{\partial}{\partial x}.\end{equation}

A straightforward calculation leads to:

$[HD_a ,HT_{b}]=ya(yb)'_y\frac{\partial}{\partial x}-y^2a'_xb\frac{\partial}{\partial y}$, $\div[HD_a ,HT_{b}]=0$,

$[HD_a, HT'_b]=xya(yb)'_y\frac{\partial}{\partial x}-xy^2a'_xb\frac{\partial}{\partial y},$ $\div[HD_a, HT'_b]=ay(yb)'_{y}.$

$\tilde H T'_b=x\frac{x(s-x)-p}{y}b\frac{\partial}{\partial x}+(s-2x)xb\frac{\partial}{\partial y}.$

Putting $b=1$ we see moreover that $[[HD_a,HT_1],\tilde HT'_1]=-a(p-x(s-x))\frac{\partial}{\partial x}+v\frac{\partial}{\partial y}$ for some $v$.

\section{Density of triangular and diagonal conjugations the spectral unit ball in $\mathbb C^{2\times 2}$}

\begin{proposition}\label{densityproperty} The group generated by finite compositions of the transposition, M\"obius maps, $T_b$ and $D_a$, where $a,b$ are described above is dense in the group of holomorphic automorphisms of the spectral unit ball.
\end{proposition}

\begin{proof}
Let $\varphi$ be an automorphism of $\Omega.$ Composing it, if necessary, with the M\"obius map one may assume that $\varphi(0)=0$ (see \cite{Ran-Whi} and \cite{Edi-Zwo}). Then $\varphi'(0)$ is also an automorphism of the spectral ball. By \cite{Ran-Whi} there is an invertible matrix $a$ such that either $\varphi'(0)(x)=axa^{-1}$, $x\in \Omega$ or $\varphi'(0)(x)=ax^ta^{-1}$, $x\in \Omega.$ Losing no generality assume that the first possibility holds. Since any invertible matrix may be represented as a finite product of triangular and diagonal matrices $\varphi'(0)$ satisfies trivially the assertion.

Note that $t\mapsto \Phi_t:=t^{-1}\varphi(t\cdot)$ is a well defined mapping from $\mathbb R$ into $\aut(\Omega),$ $\Phi_0=\varphi'(0).$ Since any automorphism of $\Omega$ fixing the origin preserves the spectrum we see that $\sigma(\Phi_t(x))=\sigma(x),$ $t\in \mathbb R,$ $x\in \Omega$ ($\sigma(x)$ denotes the spectrum of a matrix $x$).

Consider the following time-dependent vector field: $$X_{t_0}(x):=\frac{d}{dt}(\Phi_t(\Phi_{t_0}^{-1}(x)))|_{t=t_0},\ x\in \Omega.$$ It is clear that $\Phi_{t_0}\circ \Phi_0^{-1}$ is obtained by integrating $X_t$ from $0$ to $t_0.$ Since $\Phi_t$ preserves the spectrum it follows that $X_{t_0}(\det x)=0$ and $X_{t_0}(\tr x)=0$ (in other words $X_{t_0}$ is othogonal to $\tr x$ and $\det x$).

We proceed as in \cite{And-Lem} and \cite{Linphd}. $X_t$ may be approximated on compact sets by integrating the time dependent vector field $X_{k/N}$ from time $k/N$ to $(k+1)/N$. It is clear that every $X_{k/N}$ may approximated by polynomial vector fields $X'$ such that $X'(\det x)=0$ and $X'(\tr x)=0$ (note that the spectral ball is a pseuduconvex balanced domain, so any holomorphic function on it may be expanded into a series of homogenous polynomials; in particular the spectral ball is a Runge domain). We shall show that $X'$ is a Lie combination of polynomial vector fields generated by diagonal and triangular conjugations and the transposition. Then the standard argument (called sometimes the \emph{Euler's method}) would imply that we could approximate $\Phi_1\circ\Phi_0^{-1}$ by compositions of the transposition and diagonal and triangular conjugations. Whence we would be able to approximate $\Phi_1$, as well.

Therefore it remains to show that $X'$ is a sum of vector fields appearing in Section~\ref{3}. Let us denote $$X'=v_1\frac{\partial}{\partial x_{11}}+v_2\frac{\partial}{\partial x_{12}}+v_3\frac{\partial}{\partial x_{21}}+ v_4\frac{\partial}{\partial x_{22}}.$$ Assumptions on $X_{t_0}$ imply that $v_4=-v_1$ and $v_1(x_{22}-x_{11})=v_2x_{21}+v_3x_{12}$ (use the orthogonaloty of $\tr x$ and $\det x$). Note that $v_1$ may be written as \begin{equation}v_1=\sum_{j=1}^nx_{12}^jf_j(x_{11},x_{22},x_{12}x_{21})+\sum_{j=1}^nx_{21}^jg_j(x_{11},x_{22},x_{12}x_{21})+ \varphi(x_{11},x_{22},x_{12}x_{21})\end{equation}
for some polynomials $f_j,g_j,\varphi\in \mathbb C[x_{11},x_{22},x_{12}x_{21}].$

It is easily seen that $\varphi(x_{11},x_{22},0)=0,$ so $\varphi(x_{11},x_{22},x_{12}x_{21})=x_{12}x_{21}\alpha(x,s,p)$ for some polynomial $\alpha.$ Adding to $X'$, if necessary, $[[HD_a,HT_1],\tilde HT'_1]$ with suitable chosen $a$ we may assume that $\varphi=0.$

Now adding vector fields of the form $[HD_a,HT'_b]$ and $[HD_a,\tilde HT'_b]$ we may assume that $\div X'=\psi(x_{12}x_{21},x_{11},x_{22})$ for some polynomial $\psi$. Again, adding to $X'$, if necessary, vector fields of the forms $[HD_a,HT_b]$ and $[HD_a,\tilde HT_b]$ we may assume that $v_1=0.$

Thus, up to adding Lie combinations of the vector fields generated by triangular and diagonal conjugations we may assume that $$X'=v_2\frac{\partial}{\partial x_{12}}+v_3\frac{\partial}{\partial x_{21}}$$ and $X'(\det x)=0$ and $\div X'=\psi(x_{12}x_{21},x_{11},x_{22}).$ The second condition means that $x_{21}v_2+x_{12}v_3=0$ so $v_2=x_{12}w,$ $v_3=-x_{21}w$ for some polynomial $w.$ In particular, $\div X'=x_{12}\frac{\partial w}{\partial x_{12}}-x_{21}\frac{\partial w}{\partial x_{21}}$. Let us write $w$ as $w=\sum x_{12}^jx_{21}^k\zeta_{j,k}(x_{11},x_{22}).$ Then it is straightforward to see that $\sum x_{12}^jx_{21}^k(j-k)\zeta_{j,k}(x_{11},x_{22})=\psi(x_{12}x_{21},x_{11},x_{22}).$ Therefore $\zeta_{j,k}=0$ whenever $j\neq k$ and $w=w(x_{12}x_{21},x_{11},x_{22})=w(x,s,p)$. In particular $X'$ is of the form $HD_a.$
\end{proof}

\section{limit of conjugations}

\begin{remark}\label{remlc} Note that the holomorphic mapping $p$ occurring in Definition~\ref{def} is defined up to a multiplication with $x\mapsto a(x)+b(x)x.$ More precisely, $p(x)xp(x)^{-1}=q(x)xq(x)^{-1}$, $x\in U$, where $p,q\in \OO (U,\MM^{-1})$ if and only if there are holomorphic functions $a,b:U\to \CC$ such that $p(x)=(a(x)+b(x)x)q(x)$ and $\det(a(x)+b(x)x)\neq0$, $x\in U.$
\end{remark}

\begin{lemma}\label{limitofconjugations}
Let $p_n\in \mathcal O(W, \mathcal M^{-1})$, $\mathcal T\subset W\subset \Omega$, $\det p_n=1$ be a sequence of holomorphic mappings such that $\varphi_n(x):=p_n(x)xp_n(x)^{-1},$ $x\in W$ is convergent locally uniformly to $\psi\in \OO (W, \Omega).$ Then there is a neighborhood $U$ of $\TT$ and diagonal mappings $a_n,b_n\in \mathcal O(U, \mathcal M)$ such that $p_n(x)(a_n(x)+xb_n(x))$ is locally uniformly convergent on $U$ to $u\in \mathcal O(U,\mathcal M^{-1})$ and $\det (a_n(x)+xb_n(x))=1.$
\end{lemma}

The proof presented below is elementary and relies upon purely analytic methods. We do not know whether the lemma would follow from more general algebraic properties. Note that the main difficulty lies in the fact that numerical methods do not work for non-cyclic matrices.

\begin{proof} To simplify the notation we will omit subscript $n$. Composing $\varphi$ with M\"obius maps from both sides (more precisely taking $m_{-a}\circ \varphi\circ m_a,$ where $m_a=m_{a,1}$) we easily see that $q(a,x):=p(m_a(x))xp(m_a(x))$ is convergent locally uniformly with respect to $(a,x)$ in a neighborhood of $\mathbb D\times \mathcal T.$ Therefore $\frac{\partial q}{\partial x}(a,x)(h)$ converges locally uniformly in an open neighborhood of $\DD \times \mathcal T$ for any $h$. Putting $x=0$ we find that $$\frac{\partial q}{\partial x}(a,0)(h)=p(a)hp(a)^{-1},$$ whence $p(a)$ is convergent locally uniformly with respect to $a\in \DD$. Therefore, replacing $\varphi$ with $$x\mapsto p^{-1}(\left( \begin{array}{cc} \tr x/2& 0\\ 0& \tr x/2\end{array}\right))\varphi(x)p^{-1}(\left( \begin{array}{cc} \tr x/2& 0\\ 0& \tr x/2\end{array}\right))$$ we may assume that $p(x)=1$ for all non-cyclic matrices $x\in \Omega.$

Put $\Omega':=\{\left( \begin{array}{cc} x_{11}& 0\\ x_{21}& x_{11}\end{array}\right):\ x_{11}\in \DD,x_{21}\in \mathbb C\}$. First we show that there are diagonal mappings $a',$ $b'$ defined on a $U'$ neighborhood of $\TT$ in $\Omega'$ such that $p(x)(a'(x)+xb'(x))$ is convergent and $\det(a'(x)+xb'(x))=1$ for $x\in \Omega'$ lying in a neighborhood of $\TT$. Multiplying $p(x)xp(x)^{-1}$ out we get $$ \left( \begin{array}{cc} x_{11}+p_{12}p_{22}x_{21}& p_{22}^2x_{21}\\ p_{12}^2x_{21}& x_{11}- p_{12}p_{22}x_{21}\end{array}\right),\quad x\in \Omega'.$$ Using the fact that $p(x)xp(x)^{-1}$ converges locally uniformly on $W\cap \Omega'$, we get that $p_{21}$ and $p_{22}$ converge uniformly on compact subsets of $W\cap \Omega'.$ Since $p_{21}\equiv 0$ on $\TT$ we get that there is $q\in \OO(W\cap \Omega')$ such that $p_{21}(x)=x_{21}q(x).$ Moreover, $p_{22}\equiv 1$ on $\TT,$ so there is a neighborhood $U_0$ of $\TT$ in $\Omega'$, uniform with respect to $n$, on which $p_{22}$ does not vanish. Put $b'(x):=-q(x)/p_{22}(x),$ $a'(x):=(p_{22}(x)-b'(x)x_{11})/p_{22}(x),$ $x\in U.$ Direct calculations show that $a'$ and $b'$ satisfy the desired claim.

\medskip

Now we shall show that  there are diagonal mappings $a'',$ $b''$ on a neighborhood $U''$ of $\TT$ in $\Omega'':=\{\left( \begin{array}{cc} x_{11}& 0\\ x_{21}& x_{22}\end{array}\right):\ x_{11},x_{22}\in \DD, x_{21}\in \mathbb C\}$ such that $p(x)(a'(x)+xb'(x))$ is convergent locally uniformly and $\det(a'(x)+xb'(x))=1$ for $x\in \Omega''$ in a small neighborhood of $\mathcal T$.
Let us consider the following projection $$j_1:\left( \begin{array}{cc} x_{11}& 0\\ x_{21}& x_{22}\end{array}\right)\mapsto \left( \begin{array}{cc} (x_{11}+x_{22})/2& 0\\ x_{21}& (x_{11} + x_{22}) /2\end{array}\right)$$ and note that $V'':=j_1^{-1}(U')$ is a neighborhood of $\TT$ in $\Omega''.$

It follows from the previous step that $$u(x):=p(j_1(x))(a'(j_1(x))+b'(j_1(x))j_1(x)),\quad x\in V''$$ is convergent locally uniformly on $V''$ and $\det u=1$ there. Therefore, replacing $\varphi$ with $x\mapsto u^{-1}(x)\varphi(x) u(x)$ we may assume that $\varphi(x)=x$ for $x\in V''\cap \Omega'=U'$. Multiplying $p(x)xp(x)^{-1}$ out we get $$\left( \begin{array}{cc} x_{11}+p_{12}p_{21}(x_{11}-x_{22})+p_{12}p_{22}x_{21}& -p_{11}p_{12}(x_{11}- x_{22})-p_{12}^2x_{21}\\ p_{21}p_{22}(x_{11}-x_{22})+p_{22}^2x_{21}& x_{22} - p_{12}p_{21}(x_{11}-x_{22}) - p_{12}p_{22}x_{21} \end{array}\right),$$ $x\in \Omega''.$ Since $p(x)xp(x)^{-1}=x$ on $U'$ we deduce from the formula above that $p_{12}=0$ and $p_{22}^2=1$ and thus $p_{11}=p_{22}=1$ on $U'.$

In particular, there is a holomorphic function $q_{12}$ on $V''$ such that $p_{12}(x)=(x_{11}-x_{22})q_{12}(x)$. Similarly, $\tilde b(x):=(p_{22}(x)-p_{11}(x))(x_{11}-x_{22})^{-1},$ $x\in V'$ is a well defined holomorphic function on $V'$. Let us put $\tilde a(x):=p_{11}(x)+ q_{12}(x)x_{21}- \tilde b(x)x_{22}.$ It is quite elementary to verify that $p(x)(\tilde a(x)+\tilde b(x)x)$ is convergent in $V''$ (actually, to check it observe that $p(x)(\tilde a(x)+\tilde b(x)x)(x_{11}-x_{22})$ converges locally uniformly). Moreover, $\det(a(x)+b(x)x)=(p_{11}(x)+q_{12}(x)x_{21})(p_{22}(x)+q_{12}(x)x_{21}),$ so it is equal to $1$ when $x\in \TT$. Therefore there is a simply-connected neighborhood $U''$ of $\TT$ in $\Omega''$ (uniform with respect to subscripts $n$), $U''\subset V''$ such that $\det(\tilde a(x)+x\tilde b(x))$ does not vanish there. Let us take the branch of the square root $s(x):=\det(\tilde a(x)+x\tilde b(x))^{1/2}$ preserving $1$. Observe that $a'':=\tilde a/s,$ $b'':=\tilde b/s$ satisfy the claim.

\medskip

Now we prove the existence of $a$ and $b$ satisfying the assertion of the lemma. Put $$j_2:\left( \begin{array}{cc} x_{11}& x_{12}\\ x_{21}& x_{22}\end{array}\right)\mapsto \left( \begin{array}{cc} x_{11}& 0\\ x_{21}& x_{22}\end{array}\right).$$ Note that $V:=\Omega\cap j_2^{-1}(U'')$ is a neighborhood of $\TT$ in $\Omega$  so we may repeat the previous reasoning: since $v:=p\circ j_2\cdot (a''\circ j_2+j_2\cdot b''\circ j_2)$ is convergent on $V$ and $\det v=1$ there, replacing $\varphi$ with $v^{-1}\varphi v$ we may assume that $\varphi(x)=x$ for $x\in U''$. Comparing the coefficients of $p(x)xp(x)^{-1}$ for $x\in U''$:
$$\left( \begin{array}{cc} p_{11}p_{22}(x_{11}-x_{22})+p_{12}p_{22}x_{21} +x_{22}& -p_{11}p_{12}(x_{11}-x_{22})-p_{12}^2x_{21}\\ p_{21}p_{22}(x_{11}-x_{22}) +-p_{21}^2 x_{12}& -p_{11}p_{22}(x_{11}-x_{22})-p_{12}p_{22}x_{21} +x_{21}\end{array}\right)$$ and the ones of $x$ we get that $p_{12}(p_{11} (x_{11}-x_{22})+p_{12}x_{21})=0,$ so by the identity principle either $p_{12}=0$ or $p_{11}(x_{11}-x_{22})+p_{12}x_{21}=0$. If the second possibility would hold, then comparing the elements lying in the first column and the first row we would get that $x_{11}=p_{11}p_{22}(x_{11}-x_{22})+p_{12}p_{22}x_{21} +x_{22}=x_{22},$ a contradiction. Therefore $p_{12}=0$
Then, in particular, $p_{12}=0$ on $U''$. Thus $p_{12}(x)=x_{12}q(x)$, $x\in V,$ for some holomorphic function $q.$ Put $b:=-q$ and $a:=p_{11}-bx_{22}.$ Then $p(x)(a(x)+b(x)x)$ is convergent locally uniformly in $V.$ In particular, $x\mapsto \det(a(x)+b(x)x)$ is convergent. Moreover $\det(a(x)+b(x)x)=1$ on $\TT,$ therefore shrinking, if necessary, $V$ and dividing $a$ and $b$ by a proper non-vanishing holomorphic mapping we finish the proof.

\end{proof}

\section{Local form of the automorphisms of the spectral unit ball}
It is well known (see e.g. \cite{Edi-Zwo}) that for any $\varphi\in \aut(\Omega)$ there is a M\"obius map $m$ such that $\sigma(\varphi(x))=\sigma(m(x)).$ Composing $\varphi$ with $m^{-1}$ we may always assume that $\varphi$ preserves the spectrum, i.e. $\sigma(\varphi(x))=\sigma(x),$ $x\in \Omega.$

As a consequence of our conisderations we get the following
\begin{theorem} Let $\varphi$ be an automorphism of $\Omega$ preserving the spectrum. Then for any $x\in \Omega$ there is $u\in \mathcal O(U,\mathcal M_{2\times 2}^{-1})$ defined in an open neighborhood $U$ of $x$ such that $\varphi(x)=u(x)xu(x)^{-1}$ on $U.$

In other words, any automorhism of $\Omega$ preserving the spectrum is a local holomorphic conjugation.
\end{theorem}

\begin{proof} It follows from Proposition~\ref{densityproperty} that there is a sequence $(p_n)\subset \OO(\Omega,\MM^{-1})$ such that $p_n(x)xp_n(x)^{-1}$ converges locally uniformly to $\varphi(x).$ It follows from Lemma~\ref{limitofconjugations} that there is a neighborhood $U$ of $\TT$ and $u\in \OO(U,\MM^{-1})$ such that $\varphi(x)=u(x)xu(x)^{-1},$ $x\in U.$

On the other hand it is well known (see \cite{Pas}) that $\varphi$ is a local conjugation on $\Omega\setminus\TT.$
\end{proof}

\begin{remark}
It is very natural to ask whether a local conjugation on a domain $U$ satisfying "nice" topological properties (for example $H^1(U,\OO)=H^2(U,\ZZ)=0$) a local holomorphic conjugation is a holomorphic conjugation. Note, that this is equivalent to finding a solution of the following problem which may be viewed as a counterpart of the meromorphic Cousin problem:

\emph{given a covering $\{\Omega_{\alpha}\}$ and $a_{\alpha\beta},b_{\alpha\beta}\in \OO(\Omega_{\alpha}\cap U_{\beta})$, $\det(a_{\alpha\beta}(x)+xb_{\alpha\beta}(x)) \neq 0,$ $x\in \Omega_{\alpha}\cap\Omega_{\beta}$, such that  $$(a_{\alpha\beta}(x)+xb_{\alpha\beta}(x))(a_{\beta\gamma} (x)+xb_{\beta\gamma}(x))= a_{\alpha\gamma}(x)+xb_{\alpha\gamma}(x),\quad x\in \Omega_{\alpha}\cap\Omega_{\beta}\cap \Omega_{\gamma}$$ and $$(a_{\alpha\beta}(x)+xb_{\alpha\beta}(x))(a_{\beta\alpha}(x)+xb_{\beta\alpha }(x))=1,\quad x\in \Omega_{\alpha}\cap\Omega_{\beta}$$ find $a_{\alpha},b_{\alpha}\in \OO (\Omega_{\alpha})$ such that $a_{\alpha\beta}(x)+xb_{\alpha\beta}(x)= (a_{\alpha}(x)+xb_{\alpha}(x)) (a_{\beta}(x)+xb_{\beta}(x))^{-1}$ on $\Omega_{\alpha}\cap \Omega_{\beta}$ and $\det(a_{\alpha}(x)+xb_{\alpha}(x))$ does not vanish on $\Omega_{\alpha}.$}

Actually, assume that the problem stated above has a solution. A local conjugation $\varphi$ gives a data for the above problem in the following way: We may locally expand $\varphi$ a a holomorphic conjugation, i.e. there are $u_{\alpha}$ and $U_{\alpha}$ such that $\varphi(x)=u_{\alpha}(x)xu_{\alpha}(x)^{-1},$ where $u_{\alpha}\in \OO (U_{\alpha},\MM^{-1})$ and $\{U_{\alpha}\}$ is an open covering of $\Omega_{\alpha}$. Then, it follows from Remark~\ref{remlc} that $u_{\alpha,\beta}:=u_{\alpha}^{-1}u_{\beta}$ are data for the problem stated above. Solving it we find that $u_{\alpha}(x)v_{\beta}(x)^{-1}= (a_{\alpha}(x)+xb_{\alpha}(x))(a_{\beta}(x)+xb_{\beta}(x))^{-1}$ on $\Omega_{\alpha}\cap\Omega_{\beta}.$ Putting $w(x):=u_{\alpha}(x)(a_{\alpha}(x)+xb_{\alpha}(x))$, $x\in \Omega_{\alpha}$ we get a well defined holomorphic mapping on $U$ such that $\varphi(x)=w(x) x w(x)^{-1}.$

On the other hand, suppose that $\{\Omega_{\alpha\beta},a_{\alpha\beta},b_{\alpha,\beta}\}$ are data for the above problem. Then, solving the second Cousin problem for matrices we get $u_{\alpha}\in \OO(\Omega_{\alpha},\MM^{-1})$ such that $a_{\alpha\beta}(x)+xb_{\alpha\beta}(x)= u_{\alpha}(x) u_{\beta}(x)$ on $\Omega_{\alpha}\cap \Omega_{\beta}.$ Putting $\varphi(x):=u_{\alpha}(x)xu_{\alpha}(x)^{-1}$ we get a local holomorphic conjugation on $U.$ If it were a holomorphic conjugation, we would get $u\in \OO(U,\MM^{-1})$ such that $\varphi(x)=u(x)xu(x)^{-1}.$ Making use of Remark~\ref{remlc} again we get $a_{\alpha},b_{\beta}$ such that $u_{\alpha}(x)=(a_{\alpha}(x) + xb_{\alpha}(x))u(x),$ $x\in \Omega_{\alpha}.$ Then it is a direct to observe that $a_{\alpha},b_{\alpha}$ solve the above problem.

\end{remark}

\end{document}